\documentclass[12pt,reqno]{amsart}
\usepackage{amsthm, amssymb, amstext, amscd, amsfonts, amsxtra, latexsym, amsmath, comment, tikz,
	mathrsfs, mathtools, esint, stmaryrd, cancel,scalerel,hyperref,booktabs}
\usepackage{geometry}
\theoremstyle{plain}
\newtheorem{theorem}{Theorem}[section]

\newtheorem{conjecture}[theorem]{Conjecture}
\newtheorem{lemma}[theorem]{Lemma}
\newtheorem{proposition}[theorem]{Proposition}
\theoremstyle{definition}

\theoremstyle{remark}
\newtheorem*{remark}{Remark}

\numberwithin{equation}{section}

\newcommand{\N}{\mathbb N}
\newcommand{\Z}{\mathbb Z}
\newcommand{\C}{\mathbb C}

\newcommand{\sgn}{\operatorname{sgn}}
\newcommand{\CT}{\operatorname{CT}}

\makeatletter
\newcommand{\vast}{\bBigg@{4}}
\newcommand{\Vast}{\bBigg@{5}}
\makeatother

\renewcommand{\b}[1]{\boldsymbol{#1}}
\newcommand{\bop}[1]{\b{\operatorname{#1}}}

\newlength{\parenheight}
\newlength{\parendepth}
\newlength{\parendrop}

\newcommand{\paren}[4]{%
	\settoheight{\parenheight}{\(#4 #2\)}%
	\settodepth{\parendepth}{\(#4 #2\)}
	\addtolength{\parendepth}{.5ex}
	\addtolength{\parenheight}{-.5ex}
	\addtolength{\parenheight}{\parendepth}
	\addtolength{\parendepth}{-.5\parenheight}
	\setlength{\parendrop}{-.5\parenheight}
	\addtolength{\parendrop}{.5ex}
	\raisebox{-\parendepth}{\(#4
		\left#1%
		\rule[\parendrop]{0pt}{\parenheight}%
		\right.\)}
	#2
	\raisebox{-\parendepth}{\(#4
		\left.%
		\rule[\parendrop]{0pt}{\parenheight}%
		\right#3\)}
}

\def\myleft#1#2\myright#3{%
	\mathchoice{%
		\paren{#1}{#2}{#3}{\displaystyle}%
	}{%
	\paren{#1}{#2}{#3}{\textstyle}%
}{%
\paren{#1}{#2}{#3}{\scriptstyle}%
}{%
\paren{#1}{#2}{#3}{\scriptscriptstyle}%
}%
}

\renewenvironment{proof}[1][Proof]{\begin{trivlist} \item[\hskip \labelsep {\bfseries #1:}]}{\qed\end{trivlist}}

\begin{document}

\title[identities and conjectures for false theta functions and characters]
{Further $q$-series identities and conjectures relating false theta functions and characters }
\author[C. Jennings-Shaffer]{Chris Jennings-Shaffer}
\address{Department of Mathematics, University of Denver, 2390 S. York St. Denver, CO 80208}
\email{christopher.jennings-shaffer@du.edu}

\author[A. Milas]{Antun Milas}
\address{Department of Mathematics and Statistics,
   SUNY-Albany,
   Albany, NY 12222, USA}
\email{amilas@albany.edu}
\thanks{A. Milas acknowledges the support from the NSF grant DMS 1601070}


\begin{abstract} 
In this short note, a companion of \cite{JM}, we discuss several families of $q$-series identities in connection to 
false and mock theta functions, characters of modules of vertex algebras, and ``sum of tails''.
\end{abstract}

\maketitle

\section{Introduction and previous work}

In our previous work \cite{JM}, motivated by character formulas of vertex algebras and superconformal indices in physics, we obtained various identities 
for false theta functions including the following elegant identity.
\begin{theorem} \label{euler-false-p} For $k \geq 1$,
\begin{equation} \label{odd}
\frac{\sum_{n \in \mathbb{Z}} {\rm sgn}(n) q^{(k+1)n^2+kn}}{(q)^{2k}_\infty}
=
\sum_{n_1,n_2,\dotsc,n_{2k-1} \geq 0} \frac{q^{\sum_{i=1}^{2k-2} n_i n_{i+1}+\sum_{i=1}^{2k-1} n_i}}{(q)_{n_1}^2 (q)_{n_2}^2 \cdots (q)_{n_{2k-1}}^2},
\end{equation}
where as usual $(a)_n=\prod_{i=0}^{n-1} (1-aq^{i})$.
\end{theorem}
We note that these identities have an odd number of summation variables. 
Interestingly, with an even number of summation variables we obtained a family of  modular
identities conjectured in \cite{CS}. 
\begin{theorem} \label{euler-modular-p}
For $k \geq 1$, 
\begin{equation} \label{even}
\frac{(q,q^{2k+2},q^{2k+3};q^{2k+3})_\infty}{(q)^{2k+1}_\infty}
=\sum_{n_1,n_2,\dotsc,n_{2k} \geq 0} \frac{q^{\sum_{i=1}^{2k-1} n_i n_{i+1}+\sum_{i=1}^{2k} n_i}}{(q)_{n_1}^2 (q)_{n_2}^2 \cdots (q)_{n_{2k}}^2}.
\end{equation}
\end{theorem} 
In a somewhat different direction, in the same paper, we also examined $q$-series identities for false theta functions with half-integral characteristics (here $k \in \mathbb{N}$ and $\epsilon \in \{0, \frac{1}{2} \}$)
\begin{equation} \label{half}
\frac{(-q^{\frac12+\epsilon})_\infty}{(q)_\infty} \sum_{n \in \mathbb{Z}} {\rm sgn}(n) q^{\frac{1}{2}(2k+1)(n+a)^2},
\end{equation}
for some specific rational numbers $a$. We also considered related identities for certain ``shifted'' false theta series \cite[Section 3]{JM}. 

This paper aims to extend (\ref{odd}) and (\ref{even}) in a few directions. Firstly, we would like to study
related identities for the false theta functions as in (\ref{half}).  Secondly, we relax the condition on the poles 
in (\ref{odd}) and (\ref{even}) and perform a search 
for identities where the $q$-hypergeometric side takes the form   
\begin{equation} \label{master}
    \sum_{n_1,n_2,\dotsc,n_k\ge0}
\frac{q^{n_1+n_2+\dotsb+n_k+n_1n_2+n_2n_3+\dotsb+n_{k-1}n_k}}
{(q)_{n_1}^{r_1} (q)_{n_2}^{r_2} \dotsm (q)_{n_k}^{r_k} },
\end{equation}
with $k \leq \sum_{i=1}^k r_i \leq 2k$.
Lastly, we consider $q$-series identities coming from 
the formal inversion $q \mapsto q^{-1}$ of the $q$-hypergeometric term in (\ref{odd}) and (\ref{even}). This procedure is sometimes used for quantum modular forms to extend a $q$-series defined in the upper half-plane to the lower half-plane. 
 
Our paper is organized as follows. In  Sections 2 and 3 we gather several known facts.
In Section 4 we prove analogs of Theorems \ref{euler-false-p} and \ref{euler-modular-p} for false and classical theta series with half characteristics (Theorem \ref{main-4} and Proposition \ref{main-4a}). Section 5 is devoted to ``inverted identities'', under 
$q \mapsto q^{-1}$, associated to the $q$-hypergeometric series in (\ref{odd}) and (\ref{even}).  We argue that in both cases we expect modular identities. For the inverted $q$-series coming from (\ref{even}), this is proven in Proposition \ref{principal} by reduction to the character formula of a principal subspace of $A^{(1)}_{2k-1}$.
For (\ref{odd}), we expect (see Conjecture \ref{C-character}) that the resulting inverted series is modular as it is essentially the level one character of the affine vertex algebra $L_{sp(2k)}(\Lambda_0)$. We show that this is indeed true up to a cubic term (Proposition \ref{O3}). In Section 6, we study more complicated $q$-hypergeometric series of the form (\ref{master}) with $k=2$ and $k=3$. Continuing, in Section 7 we consider identities for the series (\ref{master}) with $r_i=1$ for all $i$. For $2 \leq k \leq 8$, except $k=7$, we found several interesting ``sums of tails'' type identities.  
Then in Section 8 we connect the $q$-series from Section 7  with characters of modules of principal subspaces
and infinite jet schemes. We end with a few remarks for future investigations.



\section{Quantum dilogarithm}

\subsection{Quantum dilogarithm}
As in \cite{JM}, we will approach several $q$-series identities using the quantum dilogarithm $\phi(x):=\prod_{i \geq 0} (1-q^i x)$.  Let $x$ and $y$ be non-commutative variables  such that $xy = qyx$,
then 
\begin{equation} \label{pent1}
\phi(y) \phi(x)=\phi(x) \phi(-yx) \phi(y)  , 
\end{equation}
which is Faddeev and Kashaev's pentagon identity for the quantum dilogarithm.
This identity implies that
\begin{equation}\label{pent2}
\frac{1}{\phi(x) \phi(y)}=\frac{1}{\phi(y) \phi(-yx) \phi(x)   },
\end{equation}
where $\frac{1}{\phi(x_1)\phi(x_2) \cdots \phi(x_n)}$ is understood to denote 
$\frac{1}{\phi(x_1)} \cdot \frac{1}{\phi(x_2)} \cdots \frac{1}{\phi(x_n)}$.
For its relevance in 4d/2d dualities in physics see \cite{CNV, CS} and references therein.

\section{Bailey's lemma and other known $q$-series identities}

As in \cite{JM}, we require Bailey's lemma and several standard $q$-series
identities, which we collect in this section.
A pair of sequences $(\alpha_n,\beta_n)$ is called a Bailey pair
relative to $a$ if
\begin{gather*}
\beta_n = \sum_{j=0}^n\frac{\alpha_j}{(q)_{n-j}(aq)_{n+j}}.
\end{gather*}
The $k$-fold iteration of Bailey's lemma can be found
in its entirety as Theorem 3.4 of \cite{Andrews1}. This theorem with
$k\mapsto k-1$, 
$a=q$,
$b_1=b_2=\dotsc=b_{k-1}\rightarrow\infty$,
$c_1=c_2=\dotsb=c_{k-2}=q$, $c_{k-1}=-w^{-1}q$,
$N\rightarrow\infty$, and $n_j\mapsto m_{k-j}$
states that
\begin{align}\label{EqBaileyChainSpecialized1}
&\sum_{m_1,m_2,\dotsc,m_{k-1}\ge0}
	\hspace{-1.3em}
	\frac{ (-w^{-1}q)_{m_1}(q)_{m_{k-1}} (-1)^{m_2+m_3+\dotsb+m_{k-1}}
		q^{\frac{m_1(m_1+1)}{2}+\frac{m_2(m_2+1)}{2}+\dotsb+\frac{m_{k-1}(m_{k-1}+1)}{2}}
		w^{m_1} \beta_{m_{k-1}}
	}
	{ (q)_{m_1} (q)_{m_1-m_2} \dotsm (q)_{m_{k-2}-m_{k-1}} }
\nonumber\\
&=
	\frac{(-wq)_\infty}{(q^2)_\infty}
	\sum_{n\ge0}
	\frac{(-w^{-1}q)_n (-1)^{kn} q^{\frac{(k-1)n(n+1)}{2}} w^n \alpha_n }{(-wq)_n}	
,
\end{align}
where $(\alpha_n,\beta_n)$ is any Bailey pair relative to $a=q$,
and $w\in\C$.
We require a single Bailey pair relative to $a=q$. Specifically this
is the Bailey pair $B(3)$ of Slater \cite{Slater1},
which is defined by
\begin{gather}\label{EqDefBaileyPairB3}
\alpha^{B3}_n := \frac{(-1)^n q^{\frac{n(3n+1)}{2} }(1-q^{2n+1})}{(1-q)}
,\qquad\qquad\qquad
\beta^{B3}_n := \frac{1}{(q)_n}.
\end{gather}

The additional $q$-series identity we require are as follows. We have two
identities of Euler \cite[(II.1) and (II.2)]{GasperRahman},
\begin{gather}
\label{EqEuler}
\sum_{n\ge0}\frac{z^n}{(q)_n} 
= 
	\frac{1}{(z)_\infty}
,\\
\label{EqEuler2}
\sum_{n\ge0}\frac{(-1)^nz^nq^{\frac{n(n-1)}{2}}}{(q)_n} 
= 
	(z)_\infty
.
\end{gather}
More generally, the $q$-binomial theorem
\cite[(II.3)]{GasperRahman} states that
\begin{gather}
\label{EqQBinomialTheorem}
\sum_{n\ge0}\frac{(a)_n z^n}{(q)_n} 
= 
	\frac{(az)_\infty}{(z)_\infty}
.
\end{gather}
We also need two forms of Heine's transformation \cite[(III.1) and (III.2)]{GasperRahman}, 
which are
\begin{align}
\label{EqHeinesTranformation1}
\sum_{n\ge0}\frac{(a,b)_nz^n}{(c,q)_n}
&=
	\frac{(b,az)_\infty}{(c,z)_\infty}
	\sum_{n\ge0}\frac{\left( \frac{c}{b}, z\right)_n b^n}{(az,q)_n}
,\\
\label{EqHeinesTranformation2}
\sum_{n\ge0}\frac{(a,b)_nz^n}{(c,q)_n}
&=
	\frac{\left(\frac{c}{b},bz\right)_\infty}{(c,z)_\infty}
	\sum_{n\ge0}\frac{\left( \frac{abz}{c}, b\right)_n \left(\frac{c}{b}\right)^n}{(bz,q)_n}
.
\end{align}
Lastly, we use Lemma 1 of \cite{A2} written as 
\begin{equation}\label{EqInverseJacobi}		
\frac{1}{\left(\zeta q^{\frac12},\zeta^{-1} q^{\frac12}\right)_\infty}
=
\frac{1}{(q)^2_\infty}   
\sum_{\substack{n_1\in\mathbb{Z} \\ n_2\geq |n_1|}}  (-1)^{n_1+n_2} q^{\frac{ n_2(n_2+1)}{2}-\frac{n_1^2}{2}}\zeta^{n_1}
.
\end{equation}
We note that the summation bound $n_2\geq |n_1|$ in \eqref{EqInverseJacobi}
can be replaced by $n_2\geq n_1$.

\section{Identities with half-characteristic}

In this section we extend Theorems \ref{euler-false-p} and \ref{euler-modular-p} from the introduction to half-characteristic.

\begin{proposition}\label{PropExtraIdentityStarter}
Suppose $k\ge1$
and $w\in\C$. Then
\begin{align}\label{EqConstantTermGeneral2}
&\sum_{n_1,n_2,\dotsc,n_k\ge0}
\frac{q^{n_1n_2 + n_2n_3 + \dotsb + n_{k-1}n_k + n_1+n_2+\dotsb+n_k}(-w)_{n_1}}
	{(q)_{n_1}^2(q)_{n_2}^2\dotsm (q)_{n_k}^2}
\nonumber\\
&=
	\CT_{\zeta_1,\zeta_2,\dotsc,\zeta_k}
	\frac{\phi\left(-wq^{\frac12}\zeta_1\right)}{\phi\left(q^{\frac12}\zeta_1\right)}
	\left(
		\prod_{j=2}^k \frac{1}{\phi\left(q^{\frac12}\zeta_j\right)\phi\left(q^{\frac12}\zeta_{j-1}^{-1}\right)}
	\right)
	\frac{1}{\phi\left(q^{\frac12}\zeta_k^{-1}\right)}
,
\end{align}
where the $\zeta_j$ are non-commuting variables with
$\zeta_j\zeta_{j+1}=q\zeta_{j+1}\zeta_{j}$ for $1\le j\le k-1$.
\end{proposition}
\begin{proof}
The proof is similar to that of Theorems \ref{euler-false-p} and \ref{euler-modular-p}.
We expand $\phi(-wq^{\frac12}\zeta_1)/\phi(q^{\frac12}\zeta_1)$
with the $q$-binomial theorem \eqref{EqQBinomialTheorem}
and all other products are expanded with Euler's identity \eqref{EqEuler}.
By doing so we have
\begin{align*}
&\frac{\phi\left(-wq^{\frac12}\zeta_1\right)}{\phi\left(q^{\frac12}\zeta_1\right)}
\left(
	\prod_{j=2}^k \frac{1}{\phi\left(q^{\frac12}\zeta_j\right)\phi\left(q^{\frac12}\zeta_{j-1}^{-1}\right)}
\right)
\frac{1}{\phi\left(q^{\frac12}\zeta_k^{-1}\right)}
\\
&=
	\sum_{\bop{n},\bop{m}\in\N_0^k} 	
	\frac{q^{\frac{n_1+m_1+n_2+m_2+\dotsb+n_k+m_k}{2}} (-w)_{n_1}   }
	{(q)_{n_1}(q)_{m_1}(q)_{n_2}(q)_{m_2}\dotsm (q)_{n_k}(q)_{m_k} }
	\zeta_{1}^{n_1}
	\left(
		\prod_{j=2}^k \zeta_j^{n_j} \zeta_{j-1}^{-m_{j-1}}
	\right)
	\zeta_{k}^{-m_k}
\\
&=
	\sum_{\bop{n},\bop{m}\in\N_0^k} 	
	\frac{q^{\frac{n_1+m_1+n_2+m_2+\dotsb+n_k+m_k}{2} + n_2m_1 + n_3m_2+\dotsb+ n_km_{k-1}  } 
		(-w)_{n_1}
		\zeta_1^{n_1-m_1} \zeta_2^{n_2-m_2}\dotsm \zeta_k^{n_k-m_k}
	}{(q)_{n_1}(q)_{m_1}(q)_{n_2}(q)_{m_2}\dotsm (q)_{n_k}(q)_{m_k} }
.
\end{align*}
The constant term then clearly comes from taking $m_j=n_j$ and the
proposition follows.
\end{proof}

In the lemma below, we give an intermediate identity that is required
so that we may apply Bailey's lemma.

\begin{lemma}\label{LemmaMultiSum2}Suppose $k\ge2$ and $w\in\C$. Then
\begin{align*}
&\sum_{n_1,n_2,\dotsc,n_k\ge0}
\frac{q^{n_1n_2 + n_2n_3 + \dotsb + n_{k-1}n_k + n_1+n_2+\dotsb+n_k}(-w)_{n_1}}
	{(q)_{n_1}^2(q)_{n_2}^2\dotsm (q)_{n_k}^2}
\\
&=
	\frac{1}{(q)_\infty^{k}}
	\sum_{m_1,m_2,\dotsc,m_{k-1}\ge0}
	\frac{(-1)^{m_2+\dotsb+m_{k-1}} 
		q^{\frac{m_1(m_1+1)}{2}+\frac{m_2(m_2+1)}{2}+\dotsb+\frac{m_{k-1}(m_{k-1}+1)}{2}}
		w^{m_1}(-w^{-1}q)_{m_1}}
	{(q)_{m_1}(q)_{m_1-m_2}(q)_{m_2-m_3}\dotsm (q)_{m_{k-2}-m_{k-1}}}
.
\end{align*}
\end{lemma}
\begin{proof}
We begin by reevaluating the constant term in \eqref{EqConstantTermGeneral2}
by applying \eqref{pent2} and expanding the products
with \eqref{EqQBinomialTheorem}, \eqref{EqEuler}, and 
\eqref{EqInverseJacobi}.
For convenience with the indices, we instead use
$\zeta_0,\zeta_1,\dotsc,\zeta_{k-1}$.
With this all mind, we find that
\begin{align*}
&\frac{\phi\left(-wq^{\frac12}\zeta_0\right)}{\phi\left(q^{\frac12}\zeta_0\right)}
\left(
	\prod_{j=1}^{k-1} \frac{1}{\phi\left(q^{\frac12}\zeta_j\right)\phi\left(q^{\frac12}\zeta_{j-1}^{-1}\right)}
\right)
\frac{1}{\phi\left(q^{\frac12}\zeta_{k-1}^{-1}\right)}
\\
&=
	\frac{\phi\left(-wq^{\frac12}\zeta_0\right)}
		{\phi\left(q^{\frac12}\zeta_0\right)\phi\left(q^{\frac12}\zeta_0^{-1}\right)  }
	\prod_{j=1}^{k-1} \frac{1}{\phi\left(-\zeta_j\zeta_{j-1}^{-1}\right) 
		\phi\left(q^{\frac12}\zeta_j\right)\phi\left(q^{\frac12}\zeta_{j}^{-1}\right)}
\\
&=	
	\frac{1}{(q)_\infty^{2k-2}}
	\sum_{\substack{\bop{n},\bop{m}\in\Z^{k-1} \\ r_1,r_2\in\N_0, \bop{\ell}\in\N_0^{k-1} \\ m_j\ge n_j    }}
	\frac{ (-1)^{ \sum\limits_{j=1}^{k-1} n_j + \sum\limits_{j=1}^{k-1} m_j+\sum\limits_{j=1}^{k-1} \ell_j } 
		q^{ \sum\limits_{j=1}^{k-1}  \frac{m_j(m_j+1)}{2}-\sum\limits_{j=1}^{k-1}\frac{n_j^2}{2} + \frac{r_1+r_2}{2}}		
		(-w)_{r_1}	
	}{(q)_{\ell_1}(q)_{\ell_2}\dotsm (q)_{\ell_{k-1}} (q)_{r_1}(q)_{r_2}  }
	\\&\quad\times
	\zeta_0^{r_1-r_2} 
	\prod_{j=1}^{k-1}		
	\left(\zeta_j\zeta_{j-1}^{-1}\right)^{\ell_{j}} \zeta_j^{n_j} 	
\\
&=
	\frac{1}{(q)_\infty^{2k-2}}
	\sum_{\substack{\bop{n},\bop{m}\in\Z^{k-1} \\ r_1,r_2\in\N_0, \bop{\ell}\in\N_0^{k-1} \\ m_j\ge n_j    }}
	\frac{ (-1)^{ \sum\limits_{j=1}^{k-1} n_j + \sum\limits_{j=1}^{k-1} m_j+\sum\limits_{j=1}^{k-1} \ell_j } 
		q^{ \sum\limits_{j=1}^{k-1}\frac{m_j(m_j+1)}{2} - \sum\limits_{j=1}^{k-1}\frac{n_j^2}{2}
			+ \sum\limits_{j=1}^{k-1} \frac{\ell_j(\ell_j+1)}{2}		
			+ \frac{r_1+r_2}{2}
		}
		(-w)_{r_1}			
	}{(q)_{\ell_1}(q)_{\ell_2}\dotsm (q)_{\ell_{k-1}} (q)_{r_1}(q)_{r_2}  }
	\\&\quad\times
	\zeta_0^{r_1-r_2-\ell_1} 
	\left(
		\prod_{j=1}^{k-2} \zeta_j^{n_j+\ell_{j}-\ell_{j+1}}		
	\right)
	\zeta_{k-1}^{n_k+\ell_{k-1}} 
.
\end{align*}
The constant term comes from $n_j=\ell_{j+1}-\ell_{j}$
for $1\le j\le k-2$, $n_{k-1}=-\ell_{k-1}$, and $r_2=r_1-\ell_1$. 
For the index bounds, we replace $m_{k-1}\ge n_{k-1}$ with $m_{k-1}\ge|n_{k-1}|$. 
Thus by Proposition \ref{PropExtraIdentityStarter},
\begin{align}\label{EqThing1}
&\sum_{n_1,n_2,\dotsc,n_k\ge0}
\frac{q^{n_1n_2 + n_2n_3 + \dotsb + n_{k-1}n_k + n_1+n_2+\dotsb+n_k} (-w)_{n_1} }
	{(q)_{n_1}^2(q)_{n_2}^2\dotsm (q)_{n_k}^2}
\nonumber\\
&=
	\frac{1}{(q)_\infty^{2k-2}}
	\sum_{\substack{\bop{m}\in\Z^{k-1}, \bop{\ell}\in\N_0^{k-1} \\ r\in\N_0, r\ge \ell_1  
		\\ m_{k-1}\ge \ell_{k-1} \\ m_j\ge \ell_{j+1}-\ell_{j}      }
	}
	\frac{ 
	(-1)^{ \sum\limits_{j=1}^{k-1} m_j+\sum\limits_{j=2}^{k-2} \ell_j } 
		q^{ 
			\sum\limits_{j=1}^{k-1}  \frac{m_j(m_j+1)}{2} 
			+ \frac{\ell_1^2}{2}
			+ \sum\limits_{j=2}^{k-1} \frac{\ell_j(\ell_j+1)}{2}		
			- \sum\limits_{j=1}^{k-2} \frac{(\ell_{j+1}-\ell_{j})^2}{2}
			- \frac{\ell_{k-1}^2}{2}
			+ r
		}
		(-w)_r		
	}{(q)_{\ell_1}(q)_{\ell_2}\dotsm (q)_{\ell_{k-1}} (q)_r (q)_{r-\ell_1}}  
\nonumber\\
&=
	\frac{1}{(q)_\infty^{2k-2}}
	\sum_{ r\in\N_0,  \bop{\ell},\bop{m}\in\N_0^{k-1}  }
	\frac{ 
	(-1)^{ \sum\limits_{j=1}^{k-1} m_j + \sum\limits_{j=1}^{k-1} \ell_j } 
		q^{ \sum\limits_{j=1}^{k-1}  \frac{m_j(m_j+1)}{2} 
			+ \sum\limits_{j=1}^{k-2} \frac{\ell_j(\ell_j+1)}{2}		
			+ \frac{\ell_{k-1}(\ell_{k-1}+3)}{2}  			
		}
		(-w)_{r+\ell_1}		
	}{(q)_{\ell_1}(q)_{\ell_2}\dotsm (q)_{\ell_{k-1}} (q)_r(q)_{r+\ell_1}}  
	\nonumber\\&\quad\times
	q^{	- \ell_1m_1			
		+ \sum\limits_{j=2}^{k-2} \ell_j(m_{j-1}-m_{j}) 
		+ \ell_{k-1}(m_{k-2}+m_{k-1})
		+r
	}
.
\end{align}

Due to convergence issues in certain calculations below, we view the
far right-hand side of \eqref{EqThing1} as the $x\rightarrow 1$ case of
\begin{align}\label{EqExtraIdentityDefOfF}
F(x)
&:=
	\frac{1}{(q)_\infty^{2k-2}}
	\sum_{ r\in\N_0,  \bop{\ell},\bop{m}\in\N_0^{k-1}  }
	\frac{ 
	(-1)^{ \sum\limits_{j=1}^{k-1} m_j + \sum\limits_{j=1}^{k-1} \ell_j } 
		q^{ \sum\limits_{j=1}^{k-1}  \frac{m_j(m_j+1)}{2} 
			+ \sum\limits_{j=1}^{k-2} \frac{\ell_j(\ell_j+1)}{2}		
			+ \frac{\ell_{k-1}(\ell_{k-1}+3)}{2}  			
		}
		(-xw)_{r+\ell_1}		
	}{(q)_{\ell_1}(q)_{\ell_2}\dotsm (q)_{\ell_{k-1}} (q)_r(xq)_{r+\ell_1}}  
	\nonumber\\&\quad\times
	q^{	- \ell_1m_1			
		+ \sum\limits_{j=2}^{k-2} \ell_j(m_{j-1}-m_{j}) 
		+ \ell_{k-1}(m_{k-2}+m_{k-1})
		+r
	}
.
\end{align}
We transform the inner sum on $r$ with Heine's transformation
\eqref{EqHeinesTranformation1} with $a=0$, $b= -xwq^{\ell_1}$, 
$c=xq^{\ell_1+1}$, and $z=q$, as
\begin{align*}
\sum_{r\ge0} \frac{(-xw)_{r+\ell_1} q^r}{(q)_r(xq)_{r+\ell_1}}
&=
\frac{(-xw)_{\ell_1}}{(xq)_{\ell_1}}
\sum_{r\ge0} \frac{(-xwq^{\ell_1})_{r} q^r}{(q)_r(xq^{\ell_1+1})_{r}}
=
\frac{(-xw)_{\infty}}{(xq)_\infty (q)_\infty}
\sum_{r\ge0} (-w^{-1}q)_{r} (-1)^r x^r w^r q^{\ell_1 r}
,
\end{align*}
we note that when $x=1$, the final series above is not absolutely convergent 
for all $w$ and $\ell_1$.
Thus for $|xw|<1$,
\begin{align*}
F(x)
&=
	\frac{(-xw)_\infty}{(xq)_\infty(q)_\infty^{2k-1}}
	\sum_{ r\in\N_0,  \bop{\ell},\bop{m}\in\N_0^{k-1}  }
	\frac{ 
	(-1)^{ \sum\limits_{j=1}^{k-1} m_j + \sum\limits_{j=1}^{k-1} \ell_j + r} 
		q^{ \sum\limits_{j=1}^{k-1}  \frac{m_j(m_j+1)}{2} 
			+ \sum\limits_{j=1}^{k-2} \frac{\ell_j(\ell_j+1)}{2}		
			+ \frac{\ell_{k-1}(\ell_{k-1}+3)}{2}  			
		}		
	}{(q)_{\ell_1}(q)_{\ell_2}\dotsm (q)_{\ell_{k-1}} }  
	\\&\quad\times
	q^{	- \ell_1m_1			
		+ \sum\limits_{j=2}^{k-2} \ell_j(m_{j-1}-m_{j}) 
		+ \ell_{k-1}(m_{k-2}+m_{k-1})
		+\ell_1r
	}
	x^r w^r (-w^{-1}q)_r
.
\end{align*}
We evaluate the inner sums on each $\ell_j$ with \eqref{EqEuler2}
to find that 
\begin{gather*}
\sum_{\ell_1\ge0} \frac{ (-1)^{\ell_1} q^{\ell_1(1+r-m_1)} q^{\frac{\ell_1(\ell_1-1)}{2}}  }
{(q)_{\ell_1}}
=
	(q^{1+r-m_1})_\infty
	=
	\frac{(q)_\infty}{(q)_{r-m_1}}
,
\\
\sum_{\ell_j\ge0} \frac{ (-1)^{\ell_j} q^{\ell_j(1+m_{j-1}-m_{j})} q^{\frac{\ell_j(\ell_j-1)}{2}}  }
{(q)_{\ell_j}}
=
	(q^{1+m_{j-1}-m_{j}})
	=
	\frac{(q)_\infty}{(q)_{m_{j-1}-m_{j}}}
,\qquad
\mbox{for } 2\le j\le k-2,
\\
\sum_{\ell_{k-1}\ge0} \frac{ (-1)^{\ell_{k-1}} q^{\ell_j(2+m_{k-2}+m_{k-1})} q^{\frac{\ell_{k-1}(\ell_{k-1}-1)}{2}}  }
{(q)_{\ell_{k-1}}}
=
	(q^{2+m_{k-2}+m_{k-1}})
	=
	\frac{(q)_\infty}{(q)_{m_{k-2}+m_{k-1}+1}}
.
\end{gather*}
Thus, for $|xw|<1$,
\begin{align*}
F(x)
&=
	\frac{(-xw)_\infty}{(xq)_\infty(q)_\infty^{k}}
	\sum_{ r,m_1,m_2,\dotsc,m_{k-1}\ge0 }
	\frac{ 
		(-1)^{ \sum\limits_{j=1}^{k-1}m_j + r } 
		q^{ \sum\limits_{j=1}^{k-1}\frac{m_j(m_j+1)}{2}  }
		x^r w^r (-w^{-1}q)_r		
	}{(q)_{r-m_1}(q)_{m_1-m_2}(q)_{m_2-m_3}\dotsm (q)_{m_{k-3}-m_{k-2}} (q)_{m_{k-2}+m_{k-1}+1} }  
\\
&=
	\frac{(-xw)_\infty}{(xq)_\infty(q)_\infty^{k}}
	\sum_{ r,m_1,m_2,\dotsc,m_{k-1}\ge0 }
	\frac{ 
		(-1)^{ \sum\limits_{j=1}^{k-1}m_j + r } 
		q^{ \sum\limits_{j=1}^{k-1}\frac{m_j(m_j+1)}{2} }
		x^r w^r (-w^{-1}q)_r		
	}{(q)_{r-m_1}(q)_{m_1-m_2}(q)_{m_2-m_3}\dotsm (q)_{m_{k-2}-m_{k-1}}  }  
,
\end{align*}
where the second equality follows from Heine's transformation
\eqref{EqHeinesTranformation2} with
$a\rightarrow\infty$, $b=q$, $c= q^{2+m_{k-2}}$, $z= \frac{q}{a}$,
applied to the inner sum on $m_{k-1}$.
By \eqref{EqQBinomialTheorem}, the sum on $r$ is
\begin{align*}
\sum_{r\ge0}\frac{(-1)^r x^r w^r (-w^{-1}q)_r}{(q)_{r-m_1}}
&=
	(-1)^{m_1}x^{m_1}w^{m_1}(-w^{-1}q)_{m_1}
	\sum_{r\ge0}\frac{(-1)^r x^r w^r (-w^{-1}q^{m_1+1})_r }{(q)_{r}}
\\
&=
		\frac{(-1)^{m_1}x^{m_1}w^{m_1}(-w^{-1}q)_{m_1} (xq)_\infty }
		{(xq)_{m_1}(-xw)_\infty}	
\end{align*}
and so
\begin{align*}
F(x)
&=
	\frac{1}{(q)_\infty^{k}}
	\sum_{ m_1,m_2,\dotsc,m_{k-1}\ge0 }
	\frac{ 
		(-1)^{ \sum\limits_{j=2}^{k-1}m_j } 
		q^{ \sum\limits_{j=1}^{k-1}\frac{m_j(m_j+1)}{2}   }
		x^{m_1} w^{m_1} 
		(-w^{-1}q)_{m_1}		
	}{(xq)_{m_1}(q)_{m_1-m_2}(q)_{m_2-m_3}\dotsm (q)_{m_{k-2}-m_{k-1}}  }  
.
\end{align*}
This form of $F(x)$ is well defined for exactly the same values of $x$ as
\eqref{EqExtraIdentityDefOfF} and so we find the lemma follows by setting
$x=1$.
\end{proof}

Our extension of Theorems \ref{euler-false-p} and \ref{euler-modular-p}
to the series in \eqref{half} is given here.

\begin{theorem} \label{main-4}
Suppose $k\ge2$. Then
\begin{multline*}
\sum_{n_1,n_2,\dotsc,n_k\ge0}
\frac{q^{n_1n_2 + n_2n_3 + \dotsb + n_{k-1}n_k + n_1+n_2+\dotsb+n_k}(-1)_{n_1}}
	{(q)_{n_1}^2(q)_{n_2}^2\dotsm (q)_{n_k}^2}
\\
=
\shoveright{
	\frac{(-q)_\infty}{(q)_\infty^{k+1}}
	\left(\sum_{n\ge0}+(-1)^{k}\sum_{n<0}\right)
	(-1)^{(k+1)n}
	q^{\frac{(k+2)n^2+kn}{2}}
,}
\\
\shoveleft{
\sum_{n_1,n_2,\dotsc,n_k\ge0}
\frac{q^{n_1n_2 + n_2n_3 + \dotsb + n_{k-1}n_k + n_1+n_2+\dotsb+n_k}(-q^{\frac12})_{n_1}}
	{(q)_{n_1}^2(q)_{n_2}^2\dotsm (q)_{n_k}^2}
}
\\
=
	\frac{(-q^{\frac12})_\infty}{(q)_\infty^{k+1}}
	\left(\sum_{n\ge0}+(-1)^{k}\sum_{n<0}\right)
	(-1)^{(k+1)n}
	q^{\frac{(k+2)n^2+(k+1)n}{2}}
.
\end{multline*}
\end{theorem}
\begin{proof}
We see that the series in the right-hand side of the identity in Lemma \ref{LemmaMultiSum2}
perfectly matches the statement of Bailey's lemma \eqref{EqBaileyChainSpecialized1} with
the Bailey pair in \eqref{EqDefBaileyPairB3}. By combining these statements
we have that
\begin{multline}\label{EqExtraIdentitiesFinal}
\sum_{n_1,n_2,\dotsc,n_k\ge0}
\frac{q^{n_1n_2 + n_2n_3 + \dotsb + n_{k-1}n_k + n_1+n_2+\dotsb+n_k}(-w)_{n_1}}
	{(q)_{n_1}^2(q)_{n_2}^2\dotsm (q)_{n_k}^2}
\\
=
	\frac{(-wq)_\infty}{(q)_\infty^{k+1}}
	\sum_{n\ge0}
	\frac{(-w^{-1}q)_n (-1)^{(k+1)n}  
		q^{\frac{(k+2)n^2+kn}{2}}	
		w^n (1-q^{2n+1})
	}{(-wq)_n}
.
\end{multline}

When $w=1$, the right-hand side of \eqref{EqExtraIdentitiesFinal} simplifies to 
\begin{multline}
\frac{(-q)_\infty}{(q)_\infty^{k+1}}
\sum_{n\ge0} (-1)^{(k+1)n} q^{\frac{(k+2)n^2+kn}{2}} (1-q^{2n+1})
\\
=
\frac{(-q)_\infty}{(q)_\infty^{k+1}}
\left(\sum_{n\ge0}+(-1)^{k}\sum_{n<0}\right) 
(-1)^{(k+1)n} q^{\frac{(k+2)n^2+kn}{2}} ,
\end{multline}
as claimed.
When $w=q^{\frac12}$, the right-hand side of \eqref{EqExtraIdentitiesFinal}
instead simplifies as 
\begin{multline*}
\frac{(-q^{\frac12})_\infty}{(q)_\infty^{k+1}}
\sum_{n\ge0} (-1)^{(k+1)n} q^{\frac{(k+2)n^2+(k+1)n}{2}} (1-q^{n+\frac12})
\\
=
\frac{(-q^{\frac12})_\infty}{(q)_\infty^{k+1}}
\left(\sum_{n\ge0}+(-1)^{k}\sum_{n<0}\right) 
(-1)^{(k+1)n} q^{\frac{(k+2)n^2+(k+1)n}{2}}.
\end{multline*}
\end{proof}

There is a similar identity that comes from taking $w=-q^{\frac12}$.
We state this identity in the proposition below, but omit the proof
as it is essentially the same as the other two cases.
\begin{proposition} \label{main-4a}
Suppose $k\ge2$. Then
\begin{multline*}
\sum_{n_1,n_2,\dotsc,n_k\ge0}
\frac{q^{n_1n_2 + n_2n_3 + \dotsb + n_{k-1}n_k + n_1+n_2+\dotsb+n_k}(q^{\frac12})_{n_1}}
	{(q)_{n_1}^2(q)_{n_2}^2\dotsm (q)_{n_k}^2}
\\
=
	\frac{(q^{\frac12})_\infty}{(q)_\infty^{k+1}}
	\left(\sum_{n\ge0}+(-1)^{k}\sum_{n<0}\right)
	(-1)^{kn}
	q^{\frac{(k+2)n^2+(k+1)n}{2}}
.
\end{multline*}
\end{proposition}

\section{Identities for characters}

In this section we study $q$-series identities coming from the formal inversion $q \mapsto q^{-1}$ in the $q$-hypergeometric term in (\ref{odd}) and (\ref{even}). As in \cite{JM} we make use of quantum dilogarithms to prove the following result. The same result was discussed in \cite{CNV}.
\begin{proposition} \label{principal} Let 
$${\rm ch}_W(\tau)= \sum_{n_1,n_2,\dotsc,n_{2k-1} \geq 0}  
\frac{q^{\sum_{i=1}^{2k-1} n_i^2-\sum_{i=1}^{2k-2} n_{i}n_{i+1}}}{(q)_{n_1} (q)_{n_2} \cdots (q)_{n_{2k-1} }}$$
denote the character of the principal subspaces of the level one vertex operator algebra $W(\Lambda_0)$ of type $A^{(1)}_{2k-1}$ \cite{CLM, FS}.
Then 
$$\sum_{n_1,n_2,\dotsc,n_{2k} \geq 0 } \frac{q^{\sum_{i=1}^{2k} n_i^2-\sum_{i=1}^{2k-1} n_{i}n_{i+1}}}
{(q)_{n_1}^2 (q)_n^2 \cdots (q)_{n_{2k}}^2}=\frac{1}{(q)_\infty^{2k}} {\rm ch}_{W}(\tau).$$
Moreover, after multiplication by $q^a$ for some $a \in \mathbb{Q}$, this a modular form. \end{proposition}
\begin{proof}
This follows by verifying that
\begin{align*}
&\sum_{n_1,\dotsc,n_k \geq 0 } 
	\frac{ q^{n_1^2+n_2^2+\dotsb+n_k^2 - n_1n_2-n_2n_3-\dotsb-n_{k-1}n_k}}
	{(q)_{n_1}^2 (q)_{n_2}^2 \dotsm (q)_{n_k}^2}
\\
&=
\CT_{\zeta_1,\zeta_2,\dotsc,\zeta_k}
	\phi(q^{\frac12}\zeta_1)
	\left( 
		\prod_{j=2}^k \phi( q^{\frac12}\zeta_j) \phi( q^{\frac12}\zeta_{j-1}^{-1})
	\right)
	\phi( q^{\frac12}\zeta_k^{-1})
\\
&=
\CT_{\zeta_1,\zeta_2,\dotsc,\zeta_k}
	\phi(q^{\frac12}\zeta_1)
	\phi(q^{\frac12}\zeta_1^{-1})
	\prod_{j=2}^k \phi( -q\zeta_j \zeta_{j-1}^{-1}) 
		\phi( q^{\frac12}\zeta_{j})
		\phi( q^{\frac12}\zeta_{j}^{-1})
\\
&=
\frac{1}{(q)_\infty^k}
\sum_{n_1,\dotsc,n_{k-1} \geq 0}  
	\frac{ q^{n_1^2+n_2^2+\dotsb+n_{k-1}^2 - n_1n_2-n_2n_3-\dotsb-n_{k-2}n_{k-1}}}
	{(q)_{n_1} (q)_{n_2} \dotsm (q)_{n_{k-1}}}
,
\end{align*}
where the $\zeta_i$ are non-commuting variables with $\zeta_i\zeta_{i+1} = q^{-1}\zeta_{i+1}\zeta_i$.

For the modularity, we first use an identity from  \cite[Theorem 2.1]{WZ}. This allows us to write ${\rm ch}_W(\tau)$
as a modular Wronskian of certain theta functions \cite{BCFK} (as usual, we have to multiply by $q^a$), 
which is known to be modular with respect to some congruence subgroup \cite{Mil}.

\end{proof}

For the $q$-hypergeometric series appearing in Theorem \ref{euler-false-p} we 
expect a family of modular identities. 
We first define $F_k(q)$ by letting 
\begin{equation} \label{C-series}
\sum_{n_1,n_2,\dotsc,n_{2k+1} \geq 0 } \frac{q^{\sum_{i=1}^{2k+1} n_i^2-\sum_{i=1}^{2k} n_{i} n_{i+1} }}
{(q)_{n_1}^2 (q)_{n_2}^2 \cdots (q)_{n_{2k+1}}^2}
= \frac{1}{(q)_\infty^{2k+1}} F_k(q).
\end{equation}
We believe the following to be true.
\begin{conjecture} \label{C-character}
For $k\ge1$, we have
$$F_k(q)={\rm ch}[L_{\frak{sp}(2k)}(\Lambda_0)](q),$$
which is the character of the (suitably normalized) level one  affine vertex algebra 
$L_{\frak{sp}(2k)}(\Lambda_0)$ of type $C^{(1)}_{k}$.
\end{conjecture}
Now we provide evidence in support of Conjecture (\ref{C-character}).
\begin{proposition} \label{O3}
For $k\ge1$, we have
$$F_k(q)=1+(2k+1)k q+ \frac{1}{3} k(3+5k+4k^3) q^2 +O(q^3).$$
\end{proposition}
\begin{proof} We first let
$$\sum_{n_1,...,n_{2k+1} \geq 0 } \frac{q^{\sum_{i=1}^{2k+1} n_i^2-\sum_{i=1}^{2k} n_{i}n_{i+1}}}{(q)_{n_1}^2 (q)_n^2 \cdots (q)_{n_{2k+1}}^2}
=1+ d_k q+ a_k q^2 + O(q^3).$$

The linear term is clear, it simply counts the number of positive roots in the root system of type $A_{2k+1}$. 
Put differently, it counts the number of positive integral solutions of 
\begin{equation} \label{Equal.1}
\sum_{i=1}^{2k+1} n_i^2-\sum_{i=1}^{2k} n_{i}n_{i+1}=1.
\end{equation}
These solutions are exactly of the form $(n_1,...,n_{2k+1})=(0,..,0,1,...,1,0,..,0)$,
and it is easily seen that the number of them is $d_k=(2k+1)(k+1)$.

We claim that $a_k=\frac{1}{3}(1+k)(1+2k)(6+3k+2k^2)$.
For this, we first note that the contribution from the quadratic term comes from two sources,
specifically from the terms $\frac{q}{(q)_{n_1}^2 \cdots (q)^2_{n_{2k+1}}}$ and 
$\frac{q^2}{(q)_{n_1}^2 \cdots (q)^2_{n_{2k+1}}}$. As such we have to analyze non-negative integral solutions of
(\ref{Equal.1})
and of
\begin{equation} \label{Equal.2}
\sum_{i=1}^{2k+1} n_i^2-\sum_{i=1}^{2k} n_{i}n_{i+1}=2.
\end{equation}
Clearly the solutions of (\ref{Equal.1}) contribute a total of  $\sum_{0 \leq j < i \leq 2k+1} 2(i-j)=\frac{2}{3} (k+1)(1+2k)(3+2k)$
to the quadratic coefficient.
For the second equation, no solution has $n_i \geq 3$, and so we may assume $0 \leq n_i \leq 2$. 
Next we consider two types of solutions: (a) solutions with all $n_i \leq 1$, and (b) solutions with at least one $n_i=2$.
In case (a), we see that in each such solution there must be exactly two substrings of $1$s:
\begin{equation} \label{gen-sol}
(0,...,0,1,...,1,0,...,0,1,...,1,0,...,0).
\end{equation} 
For (b), each solution takes form: 
$$(0,..,0,{1,...1,{2,...,2},1,....,1},0,..,0).$$ 
In both cases the initial or terminal subsequence of 0s can be empty.
For either (a) and (b), it is easy to see combinatorially that we have precisely ${(2k+1)+1 \choose 4}$ solutions.
Therefore, altogether there are 
$$2 {2k+2 \choose 4 } + \frac{2}{3} (k+1)(1+2k)(3+2k)=\frac{1}{3}(1+k)(1+2k)(6+3k+2k^2)$$
solutions. Next we note $\frac{1}{(q)^{2k+1}_\infty}=1+(2k+1)q+(2+5k+2k^2)q^2+O(q^3)$
and we write 
$$\sum_{n_1,n_2,\dotsc,n_{2k+1} \geq 0 } 
\frac{q^{\sum_{i=1}^{2k+1} n_i^2-\sum_{i=1}^{2k} n_{i}n_{i+1}}}
{(q)_{n_1}^2 (q)_n^2 \cdots (q)_{n_{2k+1}}^2}=(1+(2k+1)q+(2+5k+2k^2)q^2+\cdots) F_k(q).$$
Expanding  $F(q)=1+a q+bq^2+O(q^3)$ and 
and solving for $a$ and $b$ gives $a=(2k+1)k$ and $b=\frac{1}{3}(3k+5k^2+4k^4)$ as claimed.
\end{proof}
Since this is in agreement with the known properties of ${\rm ch}[L_{\frak{sp}(2k)}(\Lambda_0)](\tau)$ (see \cite{Primc}), 
our conjecture is valid $O(q^3)$.

\begin{remark} After this work was completed, we became aware of a paper by A.V. Stoyanovsky \cite{Stoyanovsky} which identifies the character of the principal subspace for $A_{2k}^{(1)}$ at level one and the character of $L_{\frak{sp}(2k)}(\Lambda_0)$. This result together with the main identity in the proof of Proposition \ref{principal} implies our Conjecture \ref{C-character}. Stoyanovsky's identity was also discussed in \cite{BW} in connection to 
Hall-Littellwood polynomials.
\end{remark}

\section{Identities for Nahm-type sums with higher order poles}

Our interest is identities for $q$-hypergeometric multi-sums of the form
\begin{gather*}
F(r_1,r_2,\dotsc,r_k)
:=
\sum_{n_1,n_2,\dotsc,n_k\ge0}
\frac{q^{n_1+n_2+\dotsb+n_k+n_1n_2+n_2n_3+\dotsb+n_{k-1}n_k}}
{(q)_{n_1}^{r_1} (q)_{n_2}^{r_2} \dotsm (q)_{n_k}^{r_k} }
,
\end{gather*}
where each $r_i\ge1$ and $r_1+r_2+\dotsb+r_k\le 2k$. While these sums
are too general for us to form a single coherent conjecture, we do
see a large number of identities for small $k$. In particular, 
all of the following are either known or easy to prove:
\begin{proposition} We have,
\begin{align*}
F(1,1)
&=
	\frac{1}{(1-q)(q)_\infty}
,\qquad
F(1,2)
=
	\frac{1}{(q)_\infty^2}
,\qquad
F(1,3)
=
	\frac{\sum_{n \in \mathbb{Z}} \sgn(n) q^{2n^2+n}}{(q)^3_\infty}
,\\
F(2,2)
&=
	\frac{(q,q^4,q^5;q^5)}{(q)_\infty^3}
,\hspace{1em}
F(1,1,1)
=
	\frac{q^{-1}\left(1-(q;q)_\infty \right)}{(q)_\infty^2}
,\hspace{1em}
F(1,1,2)
=
	\frac{\sum_{n \geq 0} (-1)^n q^{n(n+3)/2}}{(q)_\infty^3}
,\\
F(1,2,1)
&=
	\frac{1}{(1-q) (q)_\infty^2}
,\qquad
F(1,2,2)
=
	\frac{1}{(q)_\infty^3}
,\qquad
F(1,3,1)
=
	\frac{1}{(q)_\infty^3}
,\\
F(1,2,3)
&=
	\frac{\sum_{n\in\Z}\hspace{-0.05em} \sgn(n) q^{2n^2+n}}
	{(q)_\infty^4}
,\hspace{0.4em}
F(1,3,2)
=
	\frac{(q,q^4,q^5;q^5)}{(q)_\infty^4}
,\hspace{0.4em}
F(1,4,1)
=
	\frac{\sum_{n \in \mathbb{Z}}\hspace{-0.05em} \sgn(n) q^{2n^2+n}}{(q)_\infty^4}
,\\
F(2,1,2)
&=
	\frac{1+2\sum_{n\ge1}(-1)^nq^{\frac{n(n+1)}{2}}}
	{(1-q)(q)_\infty^4}
,\qquad
F(2,2,2)
=
	\frac{\sum_{n \in \mathbb{Z}} \sgn(n)q^{3n^2+2n}}{(q)_\infty^4}.
\end{align*}
\end{proposition}
\begin{proof}
By \eqref{EqEuler}, the following two identities hold,
\begin{gather*}
F(1,1)
=
	\frac{1}{(1-q)(q)_\infty}
,\qquad\qquad
F(1,2)
=
	\frac{1}{(q)_\infty^2}
.
\end{gather*}
The identity for $F(2,2)$ is simply the $k=1$ case of Theorem \ref{euler-modular-p}.
For $F(1,3)$, we begin with \eqref{EqEuler} and find that
\begin{align*}
F(1,3)
&=
	\frac{1}{(q)_\infty}\sum_{n\ge0}\frac{q^{n}}{(q)_n^2}
	=
	\frac{1}{(q)_\infty^3}\sum_{n\ge0}(-1)^nq^{\frac{n(n+1)}{2}}
	=
	\frac{1}{(q)_\infty^3}\sum_{n\in\Z}\sgn(n)q^{n(2n+1)}
,
\end{align*}
where the second equality follows from \eqref{EqHeinesTranformation1}
with $a=b=0$ and $c=z=q$.

Additional usage of \eqref{EqEuler} then yields
\begin{flalign*}
&F(1,2,1)
= 
	\frac{F(1,1)}{(q)_\infty}
	=
	\frac{1}{(1-q)(q)_\infty^2}
,&&
F(1,2,2)
=
	\frac{F(1,2)}{(q)_\infty}
	=
	\frac{1}{(q)_\infty^3}
,\\
&F(1,3,1)
=
	\frac{F(1,2)}{(q)_\infty}
	=
	\frac{1}{(q)_\infty^3}
,&&
F(1,2,3)
=
	\frac{F(1,3)}{(q)_\infty}
	=
	\frac{1}{(q)_\infty^4} \sum_{n\in\Z}\sgn(n)q^{2n^2+n}
,\\
&F(1,3,2)
=
	\frac{F(2,2)}{(q)_\infty}
	=
	\frac{(q,q^4,q^5;q^5)_\infty}{(q)_\infty^4}
,&&
F(1,4,1)
=
	\frac{F(3,1)}{(q)_\infty}
	=
	\frac{1}{(q)_\infty^4} \sum_{n\in\Z}\sgn(n)q^{2n^2+n}
.
\end{flalign*}
The identity for $F(2,2,2)$ is Theorem \ref{euler-false-p} with $k=2$.
For $F(1,1,1)$, we begin with two applications of \eqref{EqEuler} 
and then apply \eqref{EqHeinesTranformation1} with $a=b=z=q$ and $c=0$,
which is
\begin{align*}
F(1,1,1)
&=
	\frac{1}{(q)_\infty^2}
	\sum_{n\ge0}(q)_nq^n
=
	\frac{(q^2)_\infty}{(q)_\infty^2}
	\sum_{n\ge0}\frac{q^n}{(q^2)_n}
=
	\frac{1}{(q)_\infty}
	\sum_{n\ge1}\frac{q^{n-1}}{(q)_n}
=
	\frac{q^{-1}}{(q)_\infty}
	\left(
		\frac{1}{(q)_\infty}-1
	\right)
.
\end{align*}
For $F(1,1,2)$, we begin with \eqref{EqEuler} and end with \eqref{EqHeinesTranformation1},
\begin{align*}
F(1,1,2)
&=
	\frac{1}{(q)_\infty}
	\sum_{n_2,n_3\ge0}
	\frac{q^{n_2+n_3+n_2n_3}}{(q)_{n_3}^2}
=
	\frac{1}{(q)_\infty}
	\sum_{n\ge0}
	\frac{q^{n}}{(q)_n^2(1-q^{n+1})}
\\
&=
	\frac{1}{(1-q)(q)_\infty}
	\sum_{n\ge0}
	\frac{q^{n}}{(q^2,q)_n}
=
	\frac{1}{(q)_\infty^3}
	\sum_{n\ge0}
	(-1)^n q^{\frac{n(n+3)}{2}}
.
\end{align*}
The identity for F(2,1,2) is slightly more involved in that it requires
two applications of Heine's transformation. In particular, starting with
\eqref{EqEuler},
\begin{align*}
F(2,1,2)
&=
	\sum_{n_1,n_3\ge0}
	\frac{q^{n_1+n_3}}{(q)_{n_1}^2(q)_{n_3}^2(q^{1+n_1+n_3})_\infty}
=
	\frac{1}{(q)_\infty}
	\sum_{n,m\ge0}
	\frac{(q)_{n+m} q^{n+m} }{(q)_{n}^2(q)_{m}^2}
\\
&=	
	\frac{1}{(q)_\infty}
	\sum_{n\ge0} \frac{q^n}{(q)_n}
	\sum_{m\ge0} \frac{(0,q^{n+1})_m q^m}{(q)_m^2}
\overset{\eqref{EqHeinesTranformation1}}{=}
	\frac{1}{(q)_\infty^2}
	\sum_{n,m\ge0}
	\frac{(q^{-n})_m q^{n+m+mn}}{(q)_n^2}
\\
&=
	\frac{1}{(q)_\infty^2}
	\sum_{n,m\ge0}
	\frac{(-1)^m q^{n+\frac{m(m+1)}{2}}}{(q)_n(q)_{n-m}}
=
	\frac{1}{(q)_\infty^2}
	\sum_{n,m\ge0}
	\frac{(-1)^m q^{n+\frac{m(m+3)}{2}}}{(q)_n(q)_{n+m}}
\\
&=
	\frac{1}{(q)_\infty^2}
	\sum_{m\ge0}
	\frac{(-1)^m q^{\frac{m(m+3)}{2}}}{(q)_{m}}
	\sum_{n\ge0}
	\frac{q^n}{(q^{m+1},q)_n}
\\
&\overset{\eqref{EqHeinesTranformation1}}{=}
	\frac{1}{(q)_\infty^4}
	\sum_{n,m\ge0}
	(-1)^{n+m} q^{\frac{n(n+1)}{2}+nm+\frac{m(m+3)}{2}}
.
\end{align*}
However,
\begin{align*}
&(1-q)\sum_{n,m\ge0}
	(-1)^{n+m} q^{\frac{n(n+1)}{2}+nm+\frac{m(m+3)}{2}}
\\
&=
	\sum_{\substack{n\ge1\\ m\ge0}}
	(-1)^{n+m+1} q^{\frac{n(n-1)}{2}+nm+\frac{m(m+1)}{2}}
	-
	\sum_{\substack{n\ge0\\ m\ge1}}
	(-1)^{n+m+1} q^{\frac{n(n-1)}{2}+nm+\frac{m(m+1)}{2}}
\\
&=
	\sum_{n\ge1}
	(-1)^{n+1} q^{\frac{n(n-1)}{2}}
	-
	\sum_{m\ge1}
	(-1)^{m+1} q^{\frac{m(m+1)}{2}}
=
	1+2\sum_{n\ge1}(-1)^nq^{\frac{n(n+1)}{2}}
,
\end{align*}
so that
\begin{gather*}
F(2,1,2)
=
	\frac{1}{(1-q)(q)_\infty^4}
	\left(1+2\sum_{n\ge1}(-1)^nq^{\frac{n(n+1)}{2}}\right)
.
\end{gather*}
\end{proof}
Note that $F(1,1,3)$, $F(1,1,4)$, and $F(2,1,3)$ are missing in
the identities above.

Many of the identities for $F(r_1,r_2,r_3)$ follow from identities
for $F(r_1,r_2)$. This is because \eqref{EqEuler} gives that
$F(1,r_2,\dotsc,r_k) = \frac{1}{(q)_\infty}F(r_2-1,r_3,\dotsc,r_k)$, 
and trivially $F(r_1,r_2,\dotsc,r_k)=F(r_k,\dotsc,r_2,r_1)$.
Thus each identity for $F(1,r_2,\dotsc,r_k)$ yields an identity
for $F(1,r_2+1,r_3,\dotsc,r_k)$. In the following proposition,
we record one particularly simple form of this iteration.
\begin{proposition} For $k\ge2$,
\begin{gather*}
\sum_{n_1,n_2,\dotsc,n_k \geq 0} 
\frac{q^{n_1+n_2+\dotsb+n_k+n_1n_2+n_2n_3+\dotsb+n_{k-1}n_k}}
{(q)_{n_1} (q)_{n_2}^{2} (q)_{n_3}^{2} \dotsm (q)_{n_{k-1}}^2 (q)_{n_k}}
=
	\frac{1}{(1-q)(q)_\infty^{k-1}}.
\end{gather*}
\end{proposition}

\section{Relations with Sum of Tails Identities}

Continuing from the previous section, we focus on the case $n_1=n_2=\cdots=n_k=1$, where 
$4 \leq k \leq 8$. Here we see certain Lambert series, quantum modular forms, and quasi-modular forms can appear. 
We have identities for $k=4$, $5$, and $6$, and conjectures for $k=7$ and $8$.

\subsection{$k=4$}

Using \eqref{EqEuler}, it is easy to show that
\begin{gather*}
\sum_{n_1,n_2,n_3,n_4\ge0}
\frac{q^{n_1+n_2+n_3+n_4+n_1n_2+n_2n_3+n_3n_4}}{(q)_{n_1}(q)_{n_2}(q)_{n_3}(q)_{n_4}}
=
\frac{1}{(q)_\infty^2}\sum_{n_2,n_3\ge0} q^{n_2+n_3+n_2n_3}
=
\frac{q^{-1}}{(q)_\infty^2}\sum_{n\ge1}\frac{q^n}{1-q^n}.
\end{gather*}
To see the connection with the identities for $k\ge5$ and 
sums of tails identities, we note that we can also write 
\begin{gather*}
\frac{1}{(q)_\infty}\sum_{n\ge1}\frac{q^n}{1-q^n}
=
\sum_{n \geq 0} \left(\frac{1}{(q)_\infty}-\frac{1}{(q)_n}\right)
=
\frac{1}{(q)_\infty} -1 +\sum_{n \geq 1} \left(\frac{1}{(q)_\infty}-\frac{1}{(q)_n}\right).
\end{gather*}

\subsection{$k=5$}

Again by \eqref{EqEuler}, we have
\begin{gather*}
\sum_{n_1,n_2,n_3,n_4,n_5\ge0}
\frac{q^{n_1+n_2+n_3+n_4+n_5+n_1n_2+n_2n_3+n_3n_4+n_4n_5}}{(q)_{n_1}(q)_{n_2}(q)_{n_3}(q)_{n_4}(q)_{n_5}}
=
\frac{1}{(q)_\infty^2}\sum_{n\ge0}\frac{q^n}{(q)_{n}(1-q^{n+1})^2}.
\end{gather*}
We note that the series has a straight-forward combinatorial interpretation.
In particular,
\begin{gather*}
\sum_{n\ge0}\frac{q^{n+1}}{(q)_{n}(1-q^{n+1})^2}
=
\sum_{n\ge0} t(n)q^n,
\end{gather*}
where $t(n)$ is the sum of the numbers of times that the
largest part appears in each partition of $n$, which is reminiscent
of Andrews celebrated smallest parts partition function \cite{Andrews3}.
The function $t(n)$ was studied for
its asymptotic properties in \cite{GrabnerKnopfmacher1}

It turns out this series has a much more interesting representation,
\begin{align*}
&\frac{1}{(q)_\infty^2}\sum_{n\ge0}\frac{q^n}{(q)_n(1-q^{n+1})^2}
=
	\frac{q^{-1}}{(q)_\infty^2}
	\sum_{\substack{n\ge1\\ m\ge0}}
	\frac{q^{n+nm}}{(q)_{n}}
=
	\frac{q^{-1}}{(q)_\infty^2}
	\sum_{m\ge0}
	\left(
		-1
		+
		\frac{1}{(q^{1+m})_\infty}
	\right)
\\
&=
	\frac{q^{-1}}{(q)_\infty^3}
	\sum_{m\ge0}
	\left(
		(q)_m
		-
		(q)_\infty
	\right)
=
	-\frac{q^{-1}}{2(q)_\infty^3}
	\sum_{n\ge1} 
	n\left(\tfrac{12}{n}\right)q^{\frac{n^2-1}{24}}
	-
	\frac{q^{-1}}{(q)_\infty^2}\sum_{n\ge1}
	\frac{q^n}{1-q^n}
	+
	\frac{q^{-1}}{2(q)_\infty^2}
,
\end{align*}
where the final equation follows by Theorem 2 of \cite{Zagier1}
and the discussion leading up to it. We note that upon ignoring
the factor of $\frac{q^{-1}}{(q)_\infty^3}$, this final series
was an essential component in Zagier's study \cite{Zagier1} of 
a ``strange identity'' for Kontsevich's function  and Zagier's construction
of prototypical examples of quantum modular forms \cite{Zagier2}.

\subsection{$k=6$}

In this case, we find that
\begin{gather*}
\sum_{n_1,n_2,n_3,n_4,n_5,n_6\ge0}
\frac{q^{n_1+n_2+n_3+n_4+n_5+n_6+n_1n_2+n_2n_3+n_3n_4+n_4n_5+n_5n_6}}
{(q)_{n_1}(q)_{n_2}(q)_{n_3}(q)_{n_4}(q)_{n_5}(q)_{n_6}}
=
\frac{1}{(q)_\infty^2}\sum_{n,m\ge0}\frac{q^{n+m+nm}}{(q)_{n+1}(q)_{m+1}}.
\end{gather*}
We mention in passing that the double sum appears to have a partition theoretic interpetation
and matches entry A179862 of OEIS. More interesting is how this series reduces.
Again by \eqref{EqEuler}, we have 
\begin{align*}
&\frac{1}{(q)_\infty^2}\sum_{n,m\ge0}\frac{q^{n+m+nm}}{(q)_{n+1}(q)_{m+1}}
=
	\frac{q^{-1}}{(q)_\infty^2}\sum_{n,m\ge1}\frac{q^{nm}}{(q)_{n}(q)_{m}}
=
	\frac{q^{-1}}{(q)_\infty^2}\sum_{n\ge1}
	\frac{1}{(q)_n}
	\left(	
		\frac{1}{(q^n)_{\infty}} - 1
	\right)
\\
&=
	\frac{q^{-1}}{(q)_\infty^2}\sum_{n\ge1}
	\left(
 		\frac{1}{(q)_\infty}
		-
		\frac{1}{(q)_n}
		+
 		\frac{q^n}{(1-q^n)(q)_\infty}
	\right)
=
	\frac{q^{-1}}{(q)_\infty^3}
	\sum_{n\ge1}
 		\frac{q^n}{1-q^n}
 	+
	\frac{q^{-1}}{(q)_\infty^2}\sum_{n\ge1}
	\left(
 		\frac{1}{(q)_\infty}
		-
		\frac{1}{(q)_n}
	\right)
\\
&=
	\frac{q^{-1}}{(q)_\infty^3}
	\sum_{n\ge1}
 		\frac{q^n}{1-q^n}
	-
	\frac{q^{-1}}{(q)_\infty^2}\left(\frac{1}{(q)_\infty}-1\right)	
	+
	\frac{q^{-1}}{(q)_\infty^2}\sum_{n\ge0}
	\left(
 		\frac{1}{(q)_\infty}
		-
		\frac{1}{(q)_n}
	\right)
\\
&=
	\frac{2q^{-1}}{(q)_\infty^3}
	\sum_{n\ge1}
 		\frac{q^n}{1-q^n}
	-
	\frac{q^{-1}}{(q)_\infty^3}
	+
	\frac{q^{-1}}{(q)_\infty^2}
,
\end{align*}
where the final equality follows directly from 
Theorem 2 of \cite{AndrewsJimenezUrrozOno1} with 
$a=0$ and $b=c$.

\subsection{$k=7$}
Here we have a conjectural identity:
\begin{gather*}
\sum_{n_1,n_2,n_3,n_4,n_5,n_6, n_7 \ge0}
\frac{q^{n_1+n_2+n_3+n_4+n_5+n_6+n_7+n_1n_2+n_2n_3+n_3n_4+n_4n_5+n_5n_6+n_6 n_7}}
{(q)_{n_1}(q)_{n_2}(q)_{n_3}(q)_{n_4}(q)_{n_5}(q)_{n_6}(q)_{n_7}} \\
=\frac{q^{-1}}{(1-q)(q)_\infty^4} \left(\sum_{m \geq 1} (-3m+1)(-1)^m q^{\frac{3m^2+m}{2}} 
+ \sum_{m \leq -1} (3m+2)(-1)^m q^{\frac{3m^2+m}{2}}\right).
\end{gather*}
The infinite series on the right-hand side is a quantum modular form.
We also give another conjectural identity in the form of sum of tails  
$$
\sum_{n_1,n_2,n_3,n_4,n_5,n_6, n_7 \ge0}
\frac{q^{n_1+n_2+n_3+n_4+n_5+n_6+n_7+n_1n_2+n_2n_3+n_3n_4+n_4n_5+n_5n_6+n_6 n_7}}
{(q)_{n_1}(q)_{n_2}(q)_{n_3}(q)_{n_4}(q)_{n_5}(q)_{n_6}(q)_{n_7}}$$
$$=\frac{q^{-1}}{(1-q)(q)_\infty^4} \left( -1+ \sum_{n \geq 1} \frac{q^n}{1-q^n} (q)_\infty + \sum_{n \geq 0} ((q)_n -(q)_\infty) + (q)_\infty \right).$$


\subsection{$k=8$}

Lastly we offer the following conjectural identity,
\begin{gather*}
\sum_{n_1,n_2,n_3,n_4,n_5,n_6, n_7,n_8 \ge0}
\frac{q^{n_1+n_2+n_3+n_4+n_5+n_6+n_7+n_8+n_1n_2+n_2n_3+n_3n_4+n_4n_5+n_5n_6+n_6 n_7+n_7 n_8}}
{(q)_{n_1}(q)_{n_2}(q)_{n_3}(q)_{n_4}(q)_{n_5}(q)_{n_6} (q)_{n_7} (q)_{n_8}}
\\
=
	\frac{q^{-2}}{(q)_\infty^3}\left(
		\left(\frac{1}{(q)_\infty} -1 \right)
		\sum_{n \geq 1} \frac{nq^n}{1-q^n}  
	\right).
\end{gather*}

We leave it as an open question to determine the behavior for
general $k$.

\section{Principal subspaces and infinite jet schemes}
In this part we require some familiarity with vertex algebras (especially lattice vertex algebras) and principal subspaces
as developed in \cite{FS,FF, CLM,MP}. 

We first form a lattice vertex algebra $V_L$ on the integral lattice $L=\mathbb{Z}\beta_1 + \cdots + \mathbb{Z}\beta_k$, such that $(\beta_i,\beta_{i+1})=1$, for $i=1,...,k-1$, and 
zero otherwise. This is a non-degenerate even lattice for $k$ even. For $k$ odd it is degenerate with $1$-dimensional radical subspace. For simplicity of exposition we shall ignore this degeneracy and consider only $k$ even here. We consider the principal subspace \cite{MP} $$W_L=\langle e^{\beta_1},...,e^{\beta_k} \rangle \subset V_L$$ generated by $e^{\beta_i}$. Using tools of 
vertex algebras one can show that $W_L$ admits a nice monomial basis. For instance, for $k=2$, we get
$$v=    \beta_2(-j^{(2)}_1)\cdots \beta_2 (-j^{(2)}_{n_2})   \beta_1(-j^{(1)}_1)\cdots \beta_1 (-j^{(1)}_{n_{1}}),$$
where $j^{(i)}_{k} \geq 1$ and $j^{(2)}_i > n_{1}$. Defining ${\rm deg}(v):=\sum_{i=1}^{n_1} j^{(1)}_i+\sum_{i=1}^{n_2} j^{(2)}_i$, the character of $W_L$ can be computed directly from this basis as
$$\sum_{n_1,n_2 \geq 0} \frac{q^{(n_1+1)n_2}}{(q)_{n_2}} \frac{q^{n_1}}{ (q)_{n_1}},$$
which is precisely $F(1,1)$. More generally, results from \cite{MP,P} give: 
\begin{proposition} We have $F(\underbrace{1,...,1}_{{\rm k-times}})={\rm ch}[W_L](q)$.
\end{proposition}
This formula can be interpreted in the language of infinite jet schemes. Consider an affine scheme $X$ and the $m$-th  jet scheme $J_m X $ of $X$. This system has a projective limit in the category of schemes $J_\infty X :=\displaystyle{\lim_{\leftarrow}} \ J_n X$ called the arc space of the infinite jet scheme of $X$.
Consider now (here $k \geq 2$)
$$R:=\mathbb{C}[x_1,...,x_k] /(x_1 x_2, \ldots , x_{k-1}x_k)$$
and let $X={\rm Spec} \ R$. Then the coordinate ring $J_\infty R=\mathbb{C}[J_\infty X]$ has a commutative 
vertex algebra structure (see \cite{Ar,M} for instance). For general vertex algebras we have a surjective morphism 
from $\mathbb{C}[J_\infty X]$ to ${\rm gr}(V)$ but for several examples of ``nilpotent'' vertex algebras,  as well 
as some rational vertex algebras, this map is an isomorphism \cite{Hao} (see also \cite{M, AL,HE}).
For instance, using the presentation of $W_L$ and the definition of $J_\infty R$ it is easy to see using presentation results from \cite{MP,P,Hao1,Hao} that
$$W_L \cong J_\infty R$$
as graded commutative vertex algebras. Therefore the Hilbert series of the arc space of $X$ satisfies
$$HS_q(J_\infty X )={\rm ch}[W_L](q).$$
Our results in Section 7 provide a completely new combinatorial aspect of these infinite jet schemes.

\section{Final comments}
In this section we present a few problems which need to be further addressed. 
First, there is a need for better understanding of the identities in Section 7. We still do not understand 
the nature of the $q$-series appearing on the ``product'' side. Although for $k=3$ and $k=8$ they are essentially modular,  
for $k=5$ and $k=7$ they behave as quantum modular forms, and for $k=2$, $4$, and $6$ 
they are neither modular nor quantum (or false). Further connections with the sum of tails remains unclear to us.

We would also like to gain a better understanding of the combinatorics 
behind Conjecture \ref{C-character} as there is already rich combinatorics governing monomial bases of  
basic $C_n^{(1)}$-modules \cite{Primc}.

\subsection{Further $q$-series identities}
In another direction, we can slightly modify the quadratic form in $F(r_1,...,r_k)$ by adding the term $n_{k} n_1$ 
so that the summation is over the ``circle''. This way we can produce additional interesting identities.
For instance, for $r_1=\cdots = r_k=1$ and $k=3$, we get an identity for a Ramanujan's 
fifth order mock theta function
\begin{gather*}
\sum_{n_1,n_2,n_3\ge0}
\frac{q^{n_1+n_2+n_3+n_1n_2+n_2n_3+n_3n_1}}{(q)_{n_1}(q)_{n_2}(q)_{n_3}}
= \frac{1}{(q)_\infty} \sum_{n\ge0}\frac{q^n}{(q^{n+1})_{n+1}}.
\end{gather*}
This follows directly from Euler's identity and 
$\sum_{m \geq 0}^\infty q^{m(n+1)} \frac{(q)_{n+m}}{(q)_m}=\frac{(q)_n}{(q^{n+1})_{n+1}}$ 
\cite[Theorem 3.3]{Andrews4}.
For $k=5$, we conjecture an elegant (quasi)-modular identity analogous to the $k=8$ case in Section 7,
\begin{gather*}
\sum_{n_1,n_2,n_3,n_4,n_5\ge0}
\frac{q^{n_1+n_2+n_3+n_4+n_5+n_1n_2+n_2n_3+n_3n_4+n_4n_5+n_5 n_1}}{(q)_{n_1}(q)_{n_2}(q)_{n_3}(q)_{n_4}(q)_{n_5}}
= \frac{q^{-1}}{(q)_\infty^2} \sum_{n\ge1}\frac{nq^n}{1-q^{n}}.
\end{gather*}
Since both expressions on the right-hand side are mock it would be interesting to see whether this persists in

\subsection{Identities with higher order poles}
For $r_1=\cdots = r_k=2$ and $k \geq 3$, we expect
\begin{equation*} 
\frac{\sum_{n \geq 0} (-1)^{n k} q^{\frac{k}{2}n(n+1)}}{(q)^{k}_\infty}
=\sum_{n_1,n_2,\dotsc,n_{k} \geq 0} \frac{q^{\sum_{i=1}^{k-1} n_i n_{i+1}+n_{k} n_1+\sum_{i=1}^{k} n_i}}{(q)_{n_1}^2 (q)_{n_2}^2 \cdots (q)_{n_{k}}^2},
\end{equation*}
again alternating between false identities for $k$ odd, and modular identities for $k$ even (observe,  $\sum_{n \geq 0} q^{ \frac{k}{2} n(n+1)}=\frac12 \sum_{n \in \mathbb{Z}} q^{ \frac{k}{2} n(n+1)}$). Presumably, this can be proven 
by slight adjustments along the lines of \cite[Theorem 5.5]{JM}.

\bibliographystyle{amsplain}

\end{document}